\documentclass[11pt,leqno]{article}

\title{Stationary and convergent strategies\\
 in Choquet games}

\date{April 27, 2009\\[3pt] Revised January 12, 2010}

\author{%
Fran{\c c}ois G. Dorais\\
{\normalsize \texttt{dorais@umich.edu}}
\and
Carl Mummert\\
{\normalsize \texttt{mummertc@marshall.edu}}
}

\usepackage{verbatim}
\usepackage{enumerate}
\usepackage{amsmath}
\usepackage{amssymb}
\usepackage{amsthm}
\usepackage{amscd}

\newcounter{ctr}[section]

\theoremstyle{plain}
\newtheorem{theorem}[ctr]{Theorem}
\newtheorem{proposition}[ctr]{Proposition}

\newtheorem{corollary}[ctr]{Corollary}

\newtheorem{question}[ctr]{Question}
\newtheorem{example}[ctr]{Example}

\theoremstyle{definition}
\newtheorem{definition}[ctr]{Definition}

\newcommand{\EMPTY}{\textsc{Empty}}
\newcommand{\POINT}{\textsc{Nonempty}}
\newcommand{\SCG}{\mathsf{Ch}}

\newcommand{\Strategy}{\mathfrak{S}}
\newcommand{\Basis}{\mathcal{B}}

\newcommand{\set}[1]{\lbrace#1\rbrace}
\newcommand{\seq}[1]{\langle #1\rangle}

\newcommand{\h}[1]{\hat{#1}}
\newcommand{\wh}[1]{\widehat{#1}}
\newcommand{\hx}{\h{x}}
\newcommand{\hy}{\h{y}}
\newcommand{\hz}{\h{z}}
\newcommand{\hl}{\h{l}}
\newcommand{\hU}{\wh{U}}
\newcommand{\hV}{\wh{V}}

\begin{document}

\maketitle

\begin{abstract}
  If \POINT\ has a winning strategy against \EMPTY\ in the
  Choquet game on a space, the space is said to be a
  Choquet space.  Such a winning strategy allows
  \POINT\ to consider the entire finite history of previous
  moves before making each new move; a stationary strategy only 
  permits \POINT\ to consider the previous move by \EMPTY. 
  We show that \POINT\ has a stationary winning strategy for 
  every second countable $T_1$ Choquet space. More generally, 
  \POINT\ has a stationary winning strategy for any $T_1$ Choquet
  space with an open-finite basis.
  
  We also study convergent strategies for the Choquet game,
  proving the following results. %
%
  A $T_1$ space $X$ is the open continuous 
  image of a complete metric space if and only if \POINT\ 
  has a convergent winning strategy in the Choquet game
  on~$X$.  
%
  A $T_1$ space $X$ is the
  open continuous compact image of a metric space if and only if $X$ is
  metacompact and \POINT\ has a stationary convergent strategy in
  the Choquet game on~$X$. 
%
A $T_1$ space $X$ is the
  open continuous compact image of a complete metric space if and only if $X$ is
  metacompact and \POINT\ has a stationary convergent winning strategy in
  the Choquet game on~$X$. 

\vspace{\baselineskip}
\noindent \textbf{Mathematics subject classification:} Primary
90D42, 54D20; \\
Secondary 06A10, 06B35
\end{abstract}
\newpage

\section{Introduction}
The \emph{Choquet game} is a Gale--Stewart game, denoted
$\SCG(X)$, associated to a topological space~$X$. There are
two players, \EMPTY\ and \POINT, who alternate turns for
$\omega$ rounds. On round~$i$, \EMPTY\ moves first, choosing
$x_i \in X$ and an open set $U_i$ such that $x_i \in U_i$
and, if $i \geq 1$, such that $U_i \subseteq V_{i-1}$.
Then, \POINT\ responds with an open set $V_i$ such that $x_i
\in V_i \subseteq U_i$.  After all the rounds have been played,
\EMPTY\ wins if $\bigcap_{n<\omega} V_n =
\varnothing$. Otherwise, \POINT\ wins.\footnote{%
  The names of the two players \EMPTY\ and \POINT\ vary
  widely in the literature. They are sometimes called $0$
  and $1$, or $1$ and $2$, or $\beta$ and $\alpha$,
  respectively. The Choquet game described here is sometimes
  called the \textit{strong Choquet game.}}

A \textit{winning strategy} for \POINT\ is a function $\Strategy$
which takes a partial play of the game, ending with a move
by \EMPTY, and returns an open set for \POINT\ to play, such
that \POINT\ will win any play of the game that follows the
strategy.  If \POINT\ has a winning strategy for $\SCG(X)$
then $X$ is said to be a \textit{Choquet space}.  

The Choquet game was originally applied by
Choquet~\cite{Choquet} to characterize complete
metrizability of metric spaces. It has since
been used more broadly to characterize completeness in
arbitrary spaces.  Kechris~\cite{Kechris-CDST} describes
the Choquet game, and its applications to descriptive
set theory, in detail.

We are interested in the existence of two special types of
strategies for \POINT: stationary strategies and convergent
strategies.

\begin{definition}\label{D:positional}
  A \emph{stationary strategy} for \POINT\ in the game $\SCG(X)$ is a
  strategy $\Strategy$ 
that treats each move as the first move. That is, 
  \begin{equation*}
    \Strategy(\seq{x_j,U_j}_{j \leq i}) = \Strategy(\seq{x_i,U_i})
  \end{equation*}
  for every partial play $\seq{x_j,U_j}_{j \leq i}$ of \EMPTY\
  against~$\Strategy$.\footnote{%
    Stationary strategies are also called \textit{positional} or
    \textit{memoryless} strategies, or \textit{tactics}.} 
\end{definition}

\begin{definition}
  \label{D:convergent}
  A \emph{convergent strategy} for \POINT\ in the game
  $\SCG(X)$ is a strategy $\Strategy$ such that, for any
  play $\seq{x_i, U_i, V_i}_{i < \omega}$ following $\Strategy$,
  the collection $\{V_i : i < \omega \}$ is a neighborhood
  basis for any point in $\bigcap_i V_i$.
\end{definition}
  
We do not require a stationary strategy or convergent
strategy to be a winning strategy for \POINT. It is not
hard to see that any $T_0$ space $X$ in which \POINT\ has a
convergent strategy in $\SCG(X)$ is a $T_1$ space; thus some of our
results have equivalent restatements in which ``$T_1$'' is
replaced with ``$T_0$''.

Galvin~and Telg{\'a}rsky~\cite{GalvinTelgarsky86} studied
stationary strategies for generalized Choquet
games. These games, and the results of Galvin~and
Telg{\'a}rsky, are discussed in Section~\ref{S:generalized}.

It is well known that if $X$ is a complete metric space then
\POINT\ has a convergent stationary winning strategy for
$\SCG(X)$. This is already implicit in the results of
Choquet~\cite{Choquet}, where only stationary strategies are
considered.  However, the class of spaces for which \POINT\ 
has a stationary winning strategy has not been
characterized, and the class of spaces for which \POINT\ has
a convergent strategy is also uncharacterized.

Martin~\cite{Martin-TGDT} proved
that \POINT\ has a winning strategy for every space that is
representable as the set of maximal points of a
directed-complete partial order (d.c.p.o.)\ with the Scott
topology. Such a space is said to be \textit{domain
  representable} and must always be~$T_1$.  Martin
established special cases when \POINT\ has a stationary
winning strategy, and asked whether there is always a
stationary winning strategy. Our first theorem establishes
that \POINT\ has a stationary winning strategy for every
second-countable domain representable space.

\begin{theorem}
\label{T:second-countable-strategy}
  If $X$ is a second-countable $T_1$ space and \POINT\ has a
  winning strategy in $\SCG(X)$ then \POINT\ has a
  stationary winning strategy in~$\SCG(X)$.
\end{theorem}

Mummert~and Stephan~\cite{MS-TAPS} showed that the
second-countable $T_1$ Choquet spaces are precisely the
second-countable domain representable spaces, and these have
an equivalent characterization in terms of representability
by spaces of maximal filters called \textit{MF spaces}.
Combining this result with Martin's result and
Theorem~\ref{T:second-countable-strategy} gives the
following corollary.

\begin{corollary}\label{C:represent}
  Let $X$ be a second-countable $T_1$ space. Then \POINT\
  has a stationary winning strategy in $\SCG(X)$ if and only
  if $X$ is domain representable \textup{(}equivalently, if and only
  if $X$ is representable as an MF space\textup{)}. 
\end{corollary}

Bennett, Lutzer, and~Reed~\cite{Bennett2008445} established
that a broad class of $T_3$ Choquet spaces admit stationary
winning strategies for \POINT. They asked \cite[Question
5.2]{Bennett2008445} if there is an example of a $T_3$
domain representable space such that \POINT\ does not have a
stationary winning strategy.  Corollary~\ref{C:represent}
shows that such an example, if it exists, cannot be
second countable.

The results in this paper do not assume separation beyond
the $T_1$ axiom. Thus the results here are not a consequence
of those of Bennett \textit{et al.}, because there are
second-countable Hausdorff Choquet spaces that are
nonmetrizable and thus not $T_3$. One example is the
Gandy--Harrington space.  The standard proof that the
Gandy--Harrington space is a Choquet space, as presented by
Kechris~\cite{Kechris-CDST}, does not produce a stationary
strategy for \POINT, but
Theorem~\ref{T:second-countable-strategy} implies that there
is such a strategy.

In section~\ref{S:convergent}, we study convergent
strategies for \POINT\ in the Choquet game. We obtain
several theorems:

\begin{itemize}
\item (Theorem~\ref{T:appl-oimage}) A $T_1$ space $X$ is the
  open continuous image of a complete metric space if and only if
  \POINT\ has a convergent winning strategy in~$\SCG(X)$.
  
\item (Theorem~\ref{T:appl-okimage}) A $T_1$ space $X$ is
  the open continuous compact image of a metric space if and only if
  $X$ is metacompact and \POINT\ has a stationary convergent strategy
  in~$\SCG(X)$.
  
\item (Theorem~\ref{T:appl-okcimage}) A $T_1$ space $X$ is
  the open continuous compact image of a complete metric space if and
  only if $X$ is metacompact and \POINT\ has a stationary
  convergent winning strategy in~$\SCG(X)$.
\end{itemize}
Here, a space $X$ is an \textit{open continuous compact} image of a space $Y$
if there is an open continuous surjection $f \colon Y \to X$ such
that $f^{-1}(\set{x})$ is compact for each $x \in X$.

In light of these results, it is natural to ask whether a space
$X$ is the open continuous image of a metric space if and only if
\POINT\ has a convergent strategy in $\SCG(X)$.
In Example~\ref{E:noconvergent}, we show that 
there is a first-countable $T_1$ space~$X$ such that \POINT\ 
does not have a convergent strategy in~$\SCG(X)$.  
As Ponomarev~\cite{Ponomarev60} proved that every
first-countable $T_0$ space is the open continuous image of a
metric space, this example resolves the question with a
negative answer. 

\subsection*{Special bases}

Our techniques for constructing convergent and stationary
strategies require that the space being considered has bases
with specific order properties.  The study of such order
properties is well established in the literature.

\begin{definition}
Let $\Basis$ be a basis for a topological space $X$ and, for every $x
\in X$, let $\Basis_x = \set{U \in \Basis: x \in U}$ be the induced
neighborhood basis at~$x$. Then:
\begin{itemize}
\item $\Basis$ is \emph{Noetherian}~\cite{Lindgren} if it satisfies
  the ascending chain condition. That is, every non-decreasing
  sequence of basis elements is eventually constant.
\item $\Basis$ is \emph{open-finite}~\cite{Peregudov76B}
  if $\Basis[{\supseteq U}]$ is finite for every $U \in
  \Basis$. In other words, each set in the basis has only
  finitely many supersets in the basis.
\item $\Basis$ is \emph{uniform}~\cite{Aleksandrov60} if
  $\Basis_x[{\nsubseteq U}]$ is finite for every basic neighborhood
  pair $x \in U \in \Basis$. In other words, every infinite subset of
  $\Basis_x$ is a neighborhood basis at~$x$.
\item $\Basis$ is of \emph{countable
    order}~\cite{Arhangelskii63} if $\Basis_x[{\nsubseteq
    U}]$ satisfies the descending chain condition for every
  basic neighborhood pair $x \in U \in \Basis$. In other
  words, every infinite descending chain in $\Basis_x$ is a
  neighborhood basis at~$x$.
\end{itemize}
\end{definition}

\noindent It follows immediately from definitions that every uniform basis is
open-finite, and every open-finite basis is Noetherian.
Moreover, every uniform basis is of countable order.

Some of our results make use of the following classical theorems,
which indicate the utility of uniform bases.  Lindgren~and
Nyikos~\cite{Lindgren} attribute the first of these to 
Aleksandrov.

\begin{theorem}[Aleksandrov~\cite{Aleksandrov60}]
 \label{P:uniform->metacompact}
 A space has a uniform basis if and only if it is
developable and metacompact. 
\end{theorem}

\begin{theorem}[Arhangel$'$ski{\u\i}~\cite{Arhangelskii62}]
  \label{T:arhangelskii}
  Let $X$ be a $T_1$ space. Then $X$ has a uniform basis if and only
  if $X$ is the open continuous compact image of a metric space.
\end{theorem}



The previous two theorems show that the existence of a uniform basis
cannot be demonstrated for general second-countable spaces, as there
are non-metacompact second-countable spaces, which can even be
completely Hausdorff~\cite[\#69]{SS-Cit}.  Our motivation for studying
open-finite bases is that these can be obtained in very general 
circumstances.

\begin{proposition}\label{P:T1-2ctble->open-finite}
  Every second countable $T_1$ space has a countable open-finite
  basis.
\end{proposition}

\begin{proof}
  We may assume that the space $X$ is not discrete (in which case the
  result is trivial).  Let $\seq{U_i}_{i<\omega}$ enumerate a basis of
  $X$, without repetitions, such that for every $i <
  \omega$, either  $|U_i|\geq \omega$ or $|U_i| = 1$.
  
  Inductively choose $\seq{x_i,V_i}_{i<\omega}$ in such a way that if
  $U_i$ is a singleton then $V_i = U_i = \set{x_i}$, and if $|U_i|$ is
  infinite then $x_i \in V_i = U_i \setminus\set{x_j : j <
    i}$. Thus $i < j$ implies $x_i \in V_i \setminus V_j$ except when $V_j =
  \set{x_i}$. Therefore, each $V_i$ has at most $i+1$ supersets in
  $\mathcal{V} = \set{V_i: i <\omega}$.

  If all elements of the sequence $\seq{x_i}_{i<\omega}$ are isolated
  points of $X$, then $\mathcal{V}$ itself is the required basis
  for~$X$. Otherwise, let $\seq{w_k}_{k<\omega}$ enumerate the
  non-isolated points of $X$ that occur in the sequence
  $\seq{x_i}_{i<\omega}$, each with infinitely many repetitions.  Then
  define
  \begin{equation*}
    W_k = 
    \left(
      \bigcap \set{ U_i : i \leq k \land w_k \in
        U_i}
    \right)
    \setminus
    \set{x_i : i \leq k \land x_i \neq w_k}.
  \end{equation*}
  We claim that $\Basis = \mathcal{V} \cup \mathcal{W}$ is the
  required basis of~$X$.

  We first check that $\Basis$ is indeed a basis of~$X$. It is enough
  to verify that if $x \in U_i$ then there is a $B \in \Basis$ with $x
  \in B \subseteq U_i$. If $x$ is an isolated point of $X$ or $x$ does
  not occur in $\seq{x_i}_{i<\omega}$, then $x \in V_i \subseteq
  U_i$. Otherwise, $x \in W_k \subseteq U_i$ where $k \geq i$ is such
  that $x = w_k$.

  Next, we check that every $B \in \Basis$ has finitely many supersets
  in~$\Basis$.

  \emph{Case $B = V_i$.}  We have already verified that every $V_i$
  has finitely many supersets in $\mathcal{V}$. To see that $V_i$ has
  finitely many supersets in $\mathcal{W}$, note that if $k \geq i$,
  then either $w_k \neq x_i $ and $x_i \in V_i\setminus W_k$ (hence
  $V_i \nsubseteq W_k$), else $w_k = x_i$ and $W_k \subseteq V_i$
  (hence $V_i \subseteq W_k$ implies $V_i = W_k$). Therefore $V_i$ has
  at most $i+1$ supersets in~$\mathcal{W}$.

  \emph{Case $B = W_k$.}  Let $x_i = w_k$. If $j \geq
  \max(i,k)$ then either $w_j \neq w_k$ and $w_k \in W_k \setminus
  W_j$, or $w_j = w_k$ and $W_j \subseteq W_k$. Therefore there are at
  most $\max(i,k)+1$ elements of $\mathcal{W}$ that contain $W_k$.
  Also, if $j > i$ then $x_i \in W_k \setminus V_j$ so there are at
  most $i+1$ elements of $\mathcal{V}$ that contain~$W_k$. Therefore,
  $W_k$ has finitely many supersets in~$\mathcal{B}$.
\end{proof}

\section{Generalized Choquet games}\label{S:generalized}

Generalized Choquet games on a topological space $X$ are
played exactly like the usual Choquet game on $X$, so that a
play of the game determines a descending sequence
$\seq{U_i,V_i}_{i < \omega}$ of open sets.  The only
difference lies in the way the winner is determined.  We
will be interested in games where \POINT\ wins when the
sequence $\seq{U_i, V_i}_{i<\omega}$ falls into some fixed
payoff set of descending sequences of open sets.  Thus, for
example, the original Choquet game is defined with the
payoff set containing all plays $\seq{U_i, V_i}_{i <
  \omega}$ such that $\bigcap_i V_i$ is nonempty.  The
generalized Choquet game associated with payoff set $P$ is
denoted $\SCG_P(X)$. We will often think of the payoff set
as defining a property shared by the winning plays of the
game. 

Although many instances of generalized Choquet games can be
found in the literature, Galvin and
Telg{\'a}rsky~\cite{GalvinTelgarsky86} were the first to
explicitly consider this family of games. There is not much
that one can say about $\SCG_P(X)$ for arbitrary $P$, since
these are as general as Gale--Stewart games with arbitrary
payoff sets. Thus our results will focus on classes of
properties that are well-behaved.

\begin{definition}
Let $\seq{U_i}_{i<\omega}$ and $\seq{V_i}_{i<\omega}$ be descending
sequences of open sets of a space~$X$. We write $\seq{U_i}_{i<\omega} \leq
\seq{V_i}_{i<\omega}$ if for each $V_j$ there is some $U_i$
with $U_i \subseteq V_j$. We write $\seq{U_i}_{i<\omega} \equiv \seq{V_i}_{i<\omega}$ if
$\seq{U_i}_{i<\omega} \leq \seq{V_i}_{i<\omega}$ and
$\seq{V_i}_{i<\omega} \leq \seq{U_i}_{i<\omega}$.
\end{definition}

It is immediate that ${\leq}$ is a reflexive transitive
relation and that ${\equiv}$ is an equivalence relation.

\begin{definition}
Let $P$ be a set of descending sequences of open
sets of a space $X$. Then $P$ is: 
\begin{itemize}
\item \emph{monotone} if $\seq{V_i}_{i<\omega} \in P$ and
  $\seq{U_i}_{i<\omega} \leq \seq{V_i}_{i<\omega}$ implies
  $\seq{U_i}_{i<\omega} \in P$.
\item \emph{invariant} if $\seq{V_i}_{i<\omega} \in P$ and
  $\seq{U_i}_{i<\omega} \equiv \seq{V_i}_{i<\omega}$ implies
  $\seq{U_i}_{i<\omega} \in P$.
\end{itemize}
\end{definition}

Every monotone property is invariant, but not conversely. In
any play $\seq{x_i,U_i,V_i}_{i<\omega}$ of a generalized
Choquet game, we have $\seq{V_i}_{i<\omega} \equiv
\seq{U_i}_{i<\omega} \equiv \seq{U_i,V_i}_{i<\omega}$, so
for invariant properties it makes no difference which of
these three sequences is tested to determine the outcome of
the play.

Galvin and Telg{\'a}rsky considered monotone properties,
obtaining the following general result. 

\begin{theorem}[Galvin--Telg{\'a}rsky~\cite{GalvinTelgarsky86}]
  \label{T:galvin-telgarsky}
  Let $P$ be a monotone property of descending sequences of open
  subsets of $X$. If \POINT\ has a winning strategy in $\SCG_P(X)$
  then \POINT\ has a stationary winning strategy in~$\SCG_P(X)$.
\end{theorem}

\noindent
Unfortunately, the methods of Galvin and Telg{\'a}rsky rely
heavily on monotonicity. A key example of an invariant
property that is not monotone is the property
``$\bigcap_i V_i$ is not empty'' that
defines the original Choquet game. In particular,
Theorem~\ref{T:galvin-telgarsky} cannot be applied to
Choquet games in the original sense.

In the following sections, we show that if $X$ has an open-finite basis
and $P$ is an invariant property such that \POINT\ has a winning
strategy in $\SCG_P(X)$ then \POINT\ has a stationary winning strategy
in $\SCG_P(X)$. We also show that, for an invariant property $P$ on
any space $X$, if \POINT\ has a winning strategy in $\SCG_P(X)$ then
\POINT\ has a winning strategy in $\SCG_P(X)$ that only needs to
remember the last point played by \EMPTY\ and all of the previously
played open sets, and thus does not need to know the 
other points played by~\EMPTY.

\subsection{Basic properties}

The invariant properties of descending sequences of open subsets of a
space $X$ form a complete Boolean algebra of sets, as they are closed
under arbitrary unions, arbitrary intersections, and complements.  The
monotone properties are similarly closed under arbitrary unions and
intersections, but not under complements; so the monotone properties
form a complete lattice of sets.

In this section, we study the subsets of these algebras consisting of
the invariant properties for which \POINT\ has a winning strategy and
the invariant properties for which \POINT\ has a stationary winning
strategy.  The next proposition shows that the properties $P$ for
which \POINT\ has a winning strategy in $\SCG_P(X)$ are closed under
countable intersections in the algebra of invariant properties.

\begin{proposition}\label{P:gdelta}
  Let $\seq{P(k)}_{k < \omega}$ be a sequence of invariant
  properties on a space~$X$ and let $P = \bigcap_{k<\omega}
  P(k)$. If \POINT\ has a winning strategy in $\SCG_{P(k)}(X)$
  for each $k < \omega$, then \POINT\ has a winning strategy
  in~$\SCG_{P}(X)$.
\end{proposition}

\begin{proof}
  For each $k < \omega$, let $\Strategy_k$ be a winning
  strategy for \POINT\ in $\SCG_{P_k}(X)$.  Let $\langle
  \cdot,\cdot\rangle$ be a fixed bijection from $\omega
  \times \omega$ to $\omega$. We define the strategy
  $\Strategy$ by induction as follows. Given a partial play
  $\seq{x_i,U_i,V_i}_{i \leq n}$ against $\Strategy$, write $n
  = \langle{k, m}\rangle$ and then define
  \begin{equation*}
    \Strategy\seq{x_i,U_i} 
    = \Strategy_k\seq{x_{\langle{k,j}\rangle},U_{\langle{k,j}\rangle}}_{j \leq m}.
  \end{equation*}
  For any play $\seq{x_i,U_i,V_i}_{i<\omega}$ against $\Strategy$,
  the sequence $\seq{x_{\langle{k,j}\rangle},U_{\langle
      k,j\rangle},V_{\langle{k,j}\rangle}}_{j<\omega}$ is a play
  against $\Strategy_k$. Since $\seq{V_i}_{i<\omega} \equiv
  \seq{V_{\langle{k,j}\rangle}}_{j<\omega}$ we see that
  $\seq{V_i}_{i<\omega} \in P(k)$. Since this is true for every $k <
  \omega$, we conclude that $\seq{V_i}_{i<\omega} \in P$. Therefore,
  $\Strategy$ is a winning strategy for \POINT\ in $\SCG_P(X)$.
\end{proof}

\noindent
It follows from Theorem~\ref{T:galvin-telgarsky} that the monotone
properties $P$ for which \POINT\ has a stationary winning strategy in
$\SCG_P(X)$ form a filter which is closed under countable
intersections in the lattice of monotone properties. The set of
invariant properties $P$ for which \POINT\ has a stationary winning
strategy in $\SCG_P(X)$ forms a filter in the algebra of invariant properties.

\begin{proposition}
  Let $P_1$ and $P_2$ be invariant properties. If \POINT\ has winning
  stationary strategies in $\SCG_{P_2}(X)$ and $\SCG_{P_2}(X)$, then
  \POINT\ has a stationary winning strategy in $\SCG_{P_1 \cap
    P_2}(X)$.
\end{proposition}

\begin{proof}
  Given winning stationary strategies $\Strategy_1$ and
  $\Strategy_2$ for \POINT\ in $\SCG_{P_1}(X)$ and
  $\SCG_{P_2}(X)$, respectively, define 
  $\Strategy(x,U) = \Strategy_2(x,\Strategy_1(x,U))$.  If
  $\seq{x_i,U_i,V_i}_{i<\omega}$ is a play of \EMPTY\ against
  $\Strategy$ and $W_i = \Strategy_1(x_i,U_i)$ for each $i <
  \omega$, then $\seq{x_i,W_i,V_i}_{i<\omega}$ is a play
  against $\Strategy_2$ and $\seq{x_i,U_i,W_i}_{i<\omega}$
  is a play against $\Strategy_1$. Thus
  $\seq{V_i}_{i<\omega} \in P_2$ and $\seq{W_i}_{i<\omega} \in P_1$,
  but since $\seq{V_i}_{i<\omega} \equiv
  \seq{W_i}_{i<\omega}$ it follows that
  $\seq{V_i}_{i<\omega} \in P_1$ as well.
\end{proof}

 When $X$ has an open-finite basis,
Theorem~\ref{T:open-finite-stationary} and
Proposition~\ref{P:gdelta} can be employed to show that the
filter of invariant properties for which \POINT\ has a
stationary winning strategy is closed under countable intersections.  We do not
have a full characterization of the spaces with this
property.

\begin{question}
  For what spaces $X$ is the filter of invariant properties $P$ for
  which \POINT\ has a stationary winning strategy in
  $\SCG_P(X)$ closed under countable intersections?
\end{question}

From time to time, we will find it useful to restrict the
moves of the players to some fixed basis $\Basis$ for the
space $X$. For any property $P$ of descending sequences of
open sets, we let $\SCG_P(X,\Basis)$ refer to the variant of
the generalized Choquet game on $X$ in which both players
are constrained to play open sets from the basis~$\Basis$.
Strategies for this game are similarly restricted. Such
restrictions have no impact on the determinacy of games in
which the defining property is invariant. 

\begin{proposition}\label{P:changeofbasis}
  Let $\Basis$ be a basis for the space $X$ and let $B$ be a function
  such that, for every neighborhood pair $x \in U$, we have $x \in
  B(x,U) \subseteq U$ and $B(x,U) \in \Basis$. If $P$ is an invariant
  property, then:
  \begin{enumerate}[\upshape(i)]
  \item\label{L:changeofbasis:b2g}%
    If $\Strategy$ is a winning strategy for \POINT\ in
    $\SCG_P(X,\Basis)$ then
    \begin{equation*}
      \Strategy'(\seq{x_i,U_i}_{i \leq n}) 
= \Strategy(\seq{x_i,B(x_i,U_i)}_{i \leq n})
    \end{equation*}
    defines a winning strategy for \POINT\ in~$\SCG_P(X)$.
  \item \label{L:changeofbasis:g2b}%
    If $\Strategy$ is a winning strategy for \POINT\ in $\SCG_P(X)$
    then
    \begin{equation*}
      \Strategy''(\seq{x_i,U_i}_{i \leq n}) =
      B(x_n,\Strategy(\seq{x_i,U_i}_{i \leq n}))
    \end{equation*}
    defines a winning strategy for \POINT\ in~$\SCG_P(X,\Basis)$.
  \end{enumerate}
  In each case, if the original strategy was stationary
  then so is the modified strategy.
\end{proposition}

\begin{proof}
  \emph{Ad~\eqref{L:changeofbasis:b2g}.}  Fix a play
  $\seq{x_n,U_n,V_n}_{n<\omega}$ of $\SCG_P(X)$ in which \POINT\ follows
  $\Strategy'$. By definition of $\Strategy'$, we always have $V_n =
  \Strategy\seq{x_i,\widehat{U}_i}_{i\leq n}$ where $\widehat{U}_i =
  B(x_i,U_i)$. Thus $\seq{x_n,\widehat{U}_n,V_n}$ is a play of
  $\SCG_P(X,\Basis)$ wherein \POINT\ used $\Strategy$. Since
  $\Strategy$ is winning for \POINT\, we have that
  $\seq{V_n}_{n<\omega} \in P$.
  
  \emph{Ad~\eqref{L:changeofbasis:g2b}.}  Fix a play
  $\seq{x_n,U_n,V_n}_{n<\omega}$ of $\SCG_P(X,\Basis)$ in
  which \POINT\ used $\Strategy''$. By definition of
  $\Strategy''$, we always have $V_n = B(x_n,\widehat{V}_n)$
  where $\widehat{V}_n = \Strategy\seq{x_i,U_i}_{i \leq n}$.
  Thus $\seq{x_n,U_n,\widehat{V}_n}_{n<\omega}$ is a play of
  $\SCG_P(X)$ wherein \POINT\ used $\Strategy$. Since
  $\Strategy$ is winning for \POINT, we have that
  $\seq{\widehat{V}_n}_{n<\omega} \in P$. Moreover, since
  \begin{equation*}
    \widehat{V}_{n+1} \subseteq U_{n+1} \subseteq V_n \subseteq \widehat{V}_n,
  \end{equation*}
  we have $\seq{V_n}_{n<\omega} \equiv
  \seq{\widehat{V}_n}_{n<\omega}$, which means that
  $\seq{V_n}_{n<\omega} \in P$, by the invariance of~$P$.
\end{proof}

\noindent
A similar result holds for winning strategies for~\EMPTY, but we
will have no use for that result, because we only study
winning strategies for~\POINT.

\subsection{Trace strategies}
Trace strategies allow \POINT\ to ignore all points played
by \EMPTY\ except the most recent point, and thus only
consider the sequence of open sets that have been played
before the latest move. This is a much weaker restriction on
a strategy than stationarity, allowing trace strategies to 
be obtained in more general circumstances.  We will show
that a winning trace strategy can always be found when
\POINT\ has a winning strategy in $\SCG_P(X)$, without
further assumptions on~$X$. In the
Section~\ref{S:stationary}, trace strategies serve as an
intermediate step in the path towards stationary strategies.
In the Section~\ref{S:convergent}, they are used to improve
cardinality results in circumstances when we cannot obtain
stationary strategies.

\begin{definition}
  The \emph{open trace} of a (possibly partial) play
  $\seq{x_i,U_i,V_i}_{i<n}$ of $\SCG_P(X)$ is the set
  $\set{\seq{U_i,V_i}:i<n}$ of pairs of open sets that have
  been played by the two players.
\end{definition}

\begin{definition}
A \emph{trace strategy} for \POINT\ in $\SCG_P(X)$ is a
strategy $\Strategy$ where each move for \POINT\ depends only on
\EMPTY's last move and the open trace of the previous moves.  In other
words, there is a function $\Strategy^*$ such that
\begin{equation*}
  \Strategy(\seq{x_i,U_i}_{i \leq n}) 
  = \Strategy^*(x_n,U_n,\set{\seq{U_i,V_i}:i < n}),
\end{equation*}
where as usual $V_i = \Strategy(\seq{x_j,U_j}_{j \leq i})$ for~$i <
n$.
\end{definition}

Before proving the existence of trace strategies, we need to
eliminate some strange behavior that is admissible in
general strategies for Gale--Stewart games but serves no
purpose in generalized Choquet games with invariant payoff
sets. One reason behind the definition of an open trace as
the set of previous moves, rather than the sequence of
previous moves, is to reduce difficulties caused by
possibility that the players play the same move
repeatedly.  However, additional work is required to
completely remove the effects of repetition from arbitrary
strategies.

\begin{definition}
A strategy $\Strategy$ for \POINT\ in a generalized Choquet
game is \emph{stable} if $\Strategy(x_0,U_0;\dots;x_n,U_n) = U_n$
implies that
\[
\begin{split}
  \Strategy(x_0,U_0;&\dots;x_n,U_n;x_n,U_n;\dots;x_m,U_m) \\
&  = \Strategy(x_0,U_0;\dots;x_n,U_n;\dots;x_m,U_m).
\end{split}
\]
In other words, \POINT's responses are unaffected if \EMPTY\
repeats the same move two (or more) times in a play.
\end{definition}

Every trace strategy is stable; every winning strategy for
an invariant property can be made into a stable strategy.

\begin{proposition}
  \label{P:stable-strategy}
  Let $P$ be an invariant property and let $\Basis$ be an arbitrary
  basis for $X$.  If \POINT\ has a winning strategy in
  $\SCG_P(X)$ then \POINT\ has a stable winning strategy in
  $\SCG_P(X,\Basis)$.
\end{proposition}

\begin{proof}
  By Proposition~\ref{P:changeofbasis}, we may assume that we have a
  winning strategy $\Strategy$ for \POINT\ in $\SCG_P(X,\Basis)$.  We
  will refine \POINT's original strategy $\Strategy$ in two
  phases.

  We first define the strategy $\Strategy'$ to follow
  $\Strategy$ in all cases except when $\Strategy(\seq{x_i,U_i}_{i\leq
    n}) = U_n$. When this happens, we scan ahead repeating \EMPTY's
  last move until $\Strategy$'s response is different from $U_n$. If
  this never happens, we set $\Strategy'(\seq{x_i,U_i}_{i
    \leq n}) = U_n$. If, after
  some number $r$ of repetitions, we get a different answer $U_n'$, we
  define $\Strategy'(\seq{x_i,U_i}_{i\leq n}) = U_n'$ and pretend that
  \EMPTY's $n$th move was repeated $r$ times in all future queries to
  $\Strategy$. Since $P$ is invariant, the play without repetitions is
  winning for \POINT\ if and only if the play with repetitions is
  winning for \POINT.

  Next we define $\Strategy''$ from $\Strategy'$ as follows. Whenever
  \EMPTY\ repeats a move, \POINT\ initially responds (as $\Strategy'$
  requires) with \EMPTY's last played open set. However, if \EMPTY\
  suddenly plays differently, \POINT\ collapses \EMPTY's repeated
  plays to a single play before querying $\Strategy'$ on this and all
  future rounds. Again, since $P$ is invariant, the play with
  repetitions is winning for \POINT\ if and only if the play without
  repetitions is winning for \POINT. The strategy $\Strategy''$ is a
  stable winning strategy for \POINT\ in~$\SCG_P(X)$.
\end{proof}

Our next theorem implies the existence of winning trace strategies for
\POINT\ in every Choquet space.

\begin{theorem}\label{T:trace-strategy}
  Let $P$ be an invariant property and let $\Basis$ be an arbitrary
  basis for $X$.  If \POINT\ has a winning strategy in $\SCG_P(X)$
  then \POINT\ has a winning trace strategy in~$\SCG_P(X,\Basis)$.
\end{theorem}

\begin{proof}
  Let $\Strategy$ be a stable winning strategy for \POINT\
  in~$\SCG_P(X,\Basis)$ as per
  Proposition~\ref{P:stable-strategy}. (This is the only place where
  we use the fact that $P$ is invariant.) We will use $\Strategy$ to
  construct a function $\Strategy^*$, which in turn defines
  a trace strategy
  $\Strategy_t$ (as above) with
  \begin{equation*}
    \Strategy_t(\seq{x_i,U_i}_{i \leq n}) 
    = \Strategy^*(x_n,U_n,T_n),
  \end{equation*}
  where
  \begin{equation*}
    T_n = \set{\seq{U_i,\Strategy_t(\seq{x_j,U_j}_{j \leq
          i})} : i<n}.
  \end{equation*}
  We then show that $\Strategy_t$ is a winning strategy
  for \POINT\ in~$\SCG_P(X,\Basis)$.
  
  To ensure that $\Strategy_t$ is indeed a winning strategy for
  \POINT, we will simultaneously define an auxiliary
  function $w$ mapping nonempty open traces of partial plays
  to points in~$X$. This function $w$ will have the property
  that if $\seq{x_i,U_i,V_i}_{i< \omega}$ is any play
  against~$\Strategy_t$, then $\seq{y_i,U_i,V_i}_{i<\omega}$
  is a play against~$\Strategy$, where $y_i =
  w(\set{\seq{U_j,V_j}:j \leq i})$ for~$i < \omega$.

  The definition of $\Strategy^*(x,U,T)$ proceeds by induction on~$|T|$.
  As a base case, we define $\Strategy^*(x,U,\varnothing) =
  \Strategy(x,U)$; we do not need to define~$w(\varnothing)$.

  Suppose that we have specified
  $\Strategy^*(x,U,T)$ for all open traces of size less
  than~$n$. Let $\seq{x_i,U_i}_{i \leq n}$ be a partial play for
  \EMPTY\ against~$\Strategy_t$, and let $T_n$ be the corresponding
  open trace. In order to determine
  $\Strategy_t(\seq{x_i,U_i}_{i \leq n})$, we will define $V_n =
  \Strategy^*(x_n,U_n,T_n)$ and then define $y_n = w(T_{n+1})$, where
  as above $T_{n+1} = T_n \cup  \set{\seq{U_n, V_n}}$. 
  For convenience, write $V_i = \Strategy_t(\seq{x_j,U_j}_{j
    \leq i})$ for $i < n$.  By the induction hypothesis, if we define
  $y_i = w(\set{\seq{U_j,V_j}:j \leq i})$ for $i < n$, then we also have
  $V_i = \Strategy(\seq{y_j,U_j}_{j \leq i})$ for every~$i < n$.
  Moreover, we must choose $y_n$ and $V_n$ so that $V_n =
  \Strategy(\seq{y_i,U_i}_{i \leq n})$ in order to preserve the
  induction hypothesis.

  If $|T_n| < n$, then $V_n = \Strategy^*(x_n,U_n,T_n)$ and $y_n =
  w(T_n)$ have already been defined. However, the only way that
  $|T_n|<n$ could happen is if $\seq{y_i,U_i}_{i\leq n}$ includes
  redundant moves by~\EMPTY. Since $\Strategy$ is a stable strategy,
  we must have $V_n = \Strategy(\seq{y_i,U_i}_{i\leq n})$, since this
  equality was satisfied for the play obtained by contracting all
  redundant moves from $\seq{y_i,U_i}_{i\leq n}$.

  Now suppose $|T_n| = n$. Using Zorn's Lemma, find a maximal
  set $Y \subseteq U_n$ such that the map $y \mapsto V_y$ is an
  injection, where
$V_y = \Strategy(\seq{v_i,U_i}_{i \leq n})$, $v_i = y_i$ for $i <n$, and $v_n = y$.
  We necessarily have $U_n = \bigcup_{y \in Y} V_y$. So we
  can pick $y_n \in Y$ and $V_n = V_{y_n}$ such that $x_n \in V_n$.
  The fact that $y \mapsto V_y$ is an injection guarantees that
  defining $w(T_{n+1}) = y_n$ is sound.
\end{proof}
\subsection{Stationary strategies}\label{S:stationary}
Before proving our main result on the existence of
stationary winning strategies in generalized Choquet games
on spaces with open-finite bases, we will give a general
criterion for the existence of winning stationary strategies
in generalized Choquet games on any space.  The motivating
idea is that a stationary strategy for \POINT\ in the
Choquet game should respond to a move $\seq{x, U}$ with a
neighborhood of $x$ that is very small compared to $U$.

\begin{definition}
Let $X$ be a space with a basis $\Basis$, and let $\Strategy$ be a
strategy for \POINT\ in~$\SCG_P(X,\Basis)$.  We say that
$\seq{x,U,V}$ is a \emph{good triple} for~$\Strategy$ if $U \in \Basis$, $V
\in \Basis$, $x \in V \subseteq U$, and $V$ is contained in
every response of 
$\Strategy$ to a partial play by~\EMPTY\ against $\Strategy$
ending with the move $\seq{x,U}$. That is, $\seq{x,U,V}$ is a good triple for
$\Strategy$ if for every partial play $\seq{x_i,U_i,V_i}_{i < n}$ of
$\SCG_P(X,\Basis)$ against $\Strategy$ such that $U \subseteq
\bigcap_{i<n} V_i$, we have
\begin{equation*}
  V \subseteq \Strategy(x_0,U_0;\dots;x_{n-1},U_{n-1};x,U). 
\end{equation*}
We often suppress $\Strategy$, and simply say that $\seq{x,U,V}$ is a good
triple, when \POINT's strategy is clear from context.  
\end{definition}


\begin{definition}
  We say that \emph{$\Strategy$ has enough good triples},
  relative to a given basis $\Basis$, if for every open
  neighborhood $U \in \Basis$ of a point $x$ there is an open
  neighborhood $V \in \Basis$ of $x$ such that $\seq{x,U,V}$ is a
  good triple for~$\Strategy$.
\end{definition}

If $\Strategy$ is a stationary winning strategy, then
$\seq{x,U,V}$ is a good triple if and only if $x \in V \subseteq
\Strategy(x,U)$.  Therefore, a stationary strategy always
has enough good triples. On the other hand, any winning
strategy for \POINT\ with enough good triples leads to a
stationary winning strategy for~\POINT.

\begin{proposition}
  \label{P:good-triples}
  If $P$ is an invariant property, then \POINT\ has a stationary winning
  strategy in $\SCG_P(X)$ if \textup{(}and only
  if\textup{)} there is a basis $\Basis$ for~$X$ such that \POINT\ has
  a winning strategy in $\SCG_P(X,\Basis)$ with enough good triples.
\end{proposition}

\begin{proof}
  Suppose that $\Strategy$ is a winning strategy for \POINT\
  in~$\SCG_P(X,\Basis)$ that has enough good triples. Define a
  stationary strategy $\Strategy_s$ in~$\SCG_P(X,\Basis)$ by simply
  choosing, when presented with a move $\seq{x,U}$, some $V \in \Basis$
  such that $\seq{x,U,V}$ is a good triple for~$\Strategy$.
  
  Let $\seq{x_i, U_i, V_i}_{i < \omega}$ be a play of
  $\SCG_P(X,\Basis)$ following~$\Strategy_s$. We must show
  that $\seq{V_i}_{i<\omega} \in P$. This is done by
  constructing a sequence $\seq{V'_i}_{i<\omega}$ of
  elements of $\Basis$ such that $V_{i+1} \subseteq V'_{i+1}
  \subseteq V_i$ for all~$i \in \omega$ and such that $\seq{x_i, U_i,
    V'_i}_{i<\omega}$ is a play of of $\SCG_P(X,\Basis)$
  following $\Strategy$. Since $\seq{V_i}_{i<\omega} \equiv
  \seq{V'_i}_{i<\omega}$ and $P$ is invariant, it will
  immediately follow that $\seq{V_i}_{i<\omega} \in P$.

  We proceed inductively. At round $0$,
  \EMPTY\ picks $\seq{x_0, U_0}$; so $V_0$ is chosen to complete a
  good triple.  Define $V'_0 = \Strategy(x_0, U_0)$, which means $V_0
  \subseteq V'_0$ by definition of good triple.

  Now at round $i+1$, we may assume by induction that $V_i \subseteq
  V'_i$. \EMPTY\ has played $x_{i+1} \in U_{i+1} \subseteq V_i$,
  which means that $\seq{x_{i+1},U_{i+1}}$ would be a legal move for
  \EMPTY\ in response to the partial play $ \seq{ x_k, U_k, V'_k}_{k \leq i}$ 
  in~$\SCG_P(X,\Basis)$. Now $\seq{x_{i+1}, U_{i+1},
    V_{i+1}}$ is a good triple, so we know that $V'_{i+1} =
  \Strategy(\seq{x_k,U_k}_{k \leq i+1})$ has the property that $x_{i+1} \in
  V_{i+1} \subseteq V'_{i+1} \subseteq U_{i+1} \subseteq
  V_i$.  Continuing this process through all $\omega$ rounds
  produces the desired play of $\SCG_P(X,\Basis)$.
\end{proof}

We are now prepared to prove our main result on the existence of
stationary winning strategies in generalized Choquet games. This
result applies, in particular, to the original Choquet game and to 
its variant in which \POINT\ is additionally required to
follow a convergent strategy.

\begin{theorem}
  \label{T:open-finite-stationary}
  Let $X$ be a space with an open-finite basis $\Basis$ and let $P$ be an
  invariant property of descending sequences of open
  subsets of~$X$.  If \POINT\ has a winning strategy in
  $\SCG_P(X)$ then \POINT\
  has a stationary winning strategy in~$\SCG_P(X)$.
\end{theorem}

\begin{proof}
  We will show that every trace strategy for \POINT\ in
  $\SCG_P(X,\Basis)$ has enough good triples. This is
  sufficient, because if \POINT\ has a winning strategy in
  $\SCG_P(X)$ then \POINT\ has a winning trace strategy in
  $\SCG_P(X,\Basis)$ by  Theorem~\ref{T:trace-strategy}. If this trace strategy
  has enough good triples, then \POINT\ has a stationary
  winning strategy in~$\SCG_P(X)$ by
  Proposition~\ref{P:good-triples}.

  Let $\Strategy$ be a winning trace strategy for \POINT\ in
  $\SCG_P(X,\Basis)$. 
  Let $\mathcal{T}_U$ be the set of all open traces of finite partial
  plays following $\Strategy$ for which $\seq{x,U}$ is a valid next move
  for \EMPTY.  Since $\Basis$ is open-finite, there are only finitely
  many pairs of open sets that can occur in elements of
  $\mathcal{T}_U$, because each set in the pair must be a
  superset of~$U$.  Therefore, $\mathcal{T}_U$ is finite, and hence the set
  \begin{equation*}
    W = \bigcap \set{\Strategy^*(x,U,T) : T \in \mathcal{T}_U}
  \end{equation*}
  is an open neighborhood of~$x$. If $V \in \Basis$ is a neighborhood
  of $x$ with $V \subseteq W$ then $\seq{x,U,V}$ is a good triple
  for~$\Strategy$.
\end{proof}

\noindent Theorem~\ref{T:second-countable-strategy} is an immediate
consequence of Theorem~\ref{T:open-finite-stationary} and
Proposition~\ref{P:T1-2ctble->open-finite}.

\section{Convergent strategies}\label{S:convergent}

When $X$ is a metric space, either \POINT\ or \EMPTY\ can
ensure that the intersection of open sets in a play of the
Choquet game consists of at most one point, by selecting
open sets of smaller and smaller radius as the play
progresses. We generalize this to non-metric spaces via the
notion of convergent strategies, as defined in the
introduction.
Thus, if $X$ is $T_1$, the intersection of open sets in a play of
$\SCG(X)$ following a convergent strategy contains at most
one point. However, we do not require a convergent strategy
to be a winning strategy for \POINT.  

The work of Galvin~and
Telg{\'a}rsky~\cite{GalvinTelgarsky86} can be directly applied to
study convergent strategies, as the following proposition demonstrates.

\begin{proposition}
  \label{P:convergent-stationary}
  \POINT\ has a convergent strategy in $\SCG(X)$ if and only if
  \POINT\ has a stationary convergent strategy in~$\SCG(X)$.
  \textup{(}This strategy may not be a winning strategy in~$\SCG(X)$.\textup{)}
\end{proposition}
\begin{proof}
  Define a property $P$ consisting of all descending sequences $\seq{U_i}_{i <
    \omega}$ of open sets of $X$ 
  such that $\{U_i \mid i \in \omega\}$ is a neighborhood basis for
  every point in $\bigcap_{i < \omega} U_i$. To say that \POINT\ has a
  convergent strategy for $\SCG(X)$ is exactly the same as saying that
  \POINT\ has a winning strategy in $\SCG_P(X)$. Because $P$ is a
  monotone property, Theorem~\ref{T:galvin-telgarsky} applies to
  $\SCG_P(X)$, allowing any winning strategy in $\SCG_P(X)$ to be
  converted to a stationary winning strategy for $\SCG_P(X)$, which in
  turn is a stationary convergent stategy for~$\SCG(X)$.
\end{proof}

The proposition that any open continuous image of
a Choquet space is itself a Choquet space is listed as an
exercise by \mbox{Kechris~\cite[8.16]{Kechris-CDST}.} We isolate
the proof here so that we can refer to it during the proof
of Theorem~\ref{T:appl-oimage}. 

\begin{proposition}\label{P:closed-open}
  Assume that $\Strategy_Z$ is a winning strategy for
  \POINT\ in $\SCG(Z)$ and that there is an open continuous
  surjection from $Z$ to $X$. Then there is a winning
  strategy $\Strategy_X$ for \POINT\ in $\SCG(X)$.
\end{proposition}

\begin{proof}
  Let $\Strategy_Z$ be a winning strategy for \POINT\ in
  $\SCG(Z)$ and let $f\colon Z \to X$ be an open continuous
  surjection. We inductively define a strategy $\Strategy_X$
  for \POINT\ in $\SCG(X)$. The construction uses a
  back-and-forth technique following the diagram below.
\begin{equation*}
\begin{CD}
(\hx_0, \hU_0) 
@>{\Strategy_Z}>> \hV_0 
@>>> \cdots  
@>>> \hV_{k-1} @>>> (\hx_k, \hU_k)
@>{\Strategy_Z}>> \hV_k
@>>> \cdots  
\\
@AAA @VfVV @.
@VfVV @AAA @VfVV \\
(x_0, U_0) @>{\Strategy_X}>> V_0 
@>>> \cdots
@>>> V_{k-1} @>>> (x_k, U_k) @>{\Strategy_X}>> V_k
@>>> \cdots  
\end{CD}
\end{equation*}

At round $0$, given $x_0 \in U_0 \subseteq X$,
choose some $\hx_0 \in Z$ with $f(\hx_0) = x_0$, and let
$\hU_0 = f^{-1}(U_0)$. Then $\hx_0 \in \hU_0$. Let $\hV_0 =
\Strategy_Z(\hx_0, \hU_0)$; so $\hx_0 \in \hV_0$, which means
$f(\hx_0) = x_0 \in f(\hV_0)$.  Let $V_0$ be $f(\hV_0)$. 
Because $\hV_0 \subseteq \hU_0 =
f^{-1}(U_0)$, we have $V_0 \subseteq U_0$. Thus $V_0$ is a legal
first move for \POINT\ in response to $\seq{x_0, U_0}$.

Now at stage $k > 0$, given $\seq{x_k, U_k}$, let $\hU_k =
f^{-1}(U_k) \cap \hV_{k-1}$.  We know that $x_k \in U_k \subseteq
f(\hV_{k-1})$, which means that there is some point $\hx_k \in
\hV_{k-1}$ with $f(\hx_k) = x_k$.  Because $f(\hx_k) = x_k \in
U_k$, we see that $\hx_k \in f^{-1}(U_k)$, and thus $\hx_k \in
\hU_k$. Also, $\hU_k \subseteq \hV_{k-1}$, by construction, which
means that 
$\seq{\hx_k, \hU_k}$ is a legal move for \EMPTY\ in $\SCG(Z)$ in
response to the partial play $\seq{\hx_0, \hU_0, \ldots, \hV_{k-1}}$.
Let $\hV_k = \Strategy_Z(\hx_0, \hU_0, \ldots,
\hx_k,\hU_k)$ and define $V_k = f(\hV_k)$. Because $\hV_k \subseteq
\hU_k \subseteq f^{-1} (U_k)$, we have $V_k \subseteq U_k$. Because $\hx_k \in
\hV_k$, we have $x_k \in V_k$. Thus $V_k$ is a legal move for
\POINT\ in $\SCG(X)$ for this round.
\end{proof}

Our first theorem of this section characterizes the $T_1$ spaces
for which there is a convergent strategy for \POINT\ in the
Choquet game.

\begin{theorem} \label{T:appl-oimage}
  Let $X$ be a $T_1$ space. Then $X$ is the open continuous image of a
  complete metric space if and only if \POINT\ has a
  convergent winning strategy in~$\SCG(X)$.  Moreover, the
  metric space can be taken to have the same weight as~$X$.
\end{theorem}

\begin{proof}
  For the forward direction, suppose $f \colon Z \to X$ is an open
  continuous surjection from a complete metric space
  $Z$ to a $T_1$ space~$X$.  Let $\Strategy_Z$ be a
  convergent winning trace strategy for \POINT\ in~$\SCG(Z)$ with
  the property that the open sets played by \POINT\ in any
  play following $\Strategy_Z$ have radii converging to $0$,
  and thus the sequence of points played by \EMPTY\ is a
  Cauchy sequence. The canonical winning strategy for
  \POINT\ in $\SCG(Z)$ has these properties. Construct a
  winning strategy $\Strategy_X$ for \POINT\ in $\SCG(X)$ exactly as in
  the proof of Proposition~\ref{P:closed-open}.  

  We must prove that $\Strategy_X$ is a convergent strategy.
  Let $\seq{x_k, U_k}_{k <M \omega}$ be a play of $\SCG(X)$ following
  $\Strategy_X$. Then there is a corresponding play $\seq{\hx_k,
  \hU_k}_{k < \omega}$ as defined in the construction of $\Strategy_Z$,
  and a corresponding sequence $\seq{\hV_k}_{k < \omega}$. 
  
  Because $\Strategy_Z$ is a convergent winning strategy for
  \POINT\ in $\SCG(Z)$, there is a single point $z \in
  \bigcap_i \hU_i$, which is the limit of the sequence
  $\seq{\hx_i}_{i < \omega}$.  To see that $U = \bigcap_i U_i$ is a
  singleton, suppose $y$ and $z$ are distinct points of~$U$.
  Then there is a sequence $\seq{\hy_i}_{i < \omega}$ such that $\hy_i
  \in \hU_i$ and $f(\hy_i) = y$, and a sequence
  $\seq{\hz_i}_{i < \omega}$ such that $\hz_i \in \hU_i$ and $f(\hz_i) =
  z$. Now, because the radii of $\seq{\hU_i}_{i < \omega}$ converge
  to~$0$, and $Z$ is a complete space, both $\seq{\hy_i}_{i
    < \omega}$
  and $\seq{\hz_i}_{i < \omega}$ are convergent, and have the same
  limit~$\hl$. Now let $W \subseteq X$ be an open
  neighborhood of $z$ with $y \not \in W$. Then $\hl \in
  f^{-1}(W)$ and so $\seq{\hy_i}_{i < \omega}$ is eventually in
  $f^{-1}(W)$, which is impossible because $y \not \in W$
  but $f(\hy_i) = y$.
  
  This shows that $U$ contains a single point $z = f(\hz)$.
  Now let $W' \subseteq X$ be any open neighborhood of~$z$.
  Then $f^{-1}(W')$ is an open neighborhood of~$\hz$, and so
  $\hV_k \subseteq f^{-1}(W')$ for some $k$. Then $z \in V_k
  \subseteq W'$; this shows that $\Strategy_X$ is a
  convergent winning strategy for \POINT.
  
  For the converse, assume that $X$ is a $T_1$ space and
  that \POINT\ has a convergent winning strategy $\Strategy$
  in~$\SCG(X)$. Note that the property that a play
  of $\SCG(X)$ is convergent and winning is an invariant
  property.  Thus, by Theorem~\ref{T:trace-strategy}, we may
  assume $\Strategy$ is a trace strategy.
 
  Let $\Basis$ be any basis for $X$ and let $S \subseteq
  \prod_{n<\omega} \Basis^2$ be the set of all descending
  sequences $\seq{U_n, V_n}_{n<\omega}$ of pairs of elements
  of $\Basis$ such that for every $n < \omega$ there is some $x
  \in U_{n}$ such that
\[
  U_{n+1} \subseteq V_n = \Strategy(U_n, x, \set{\seq{U_i,V_i} : i < n}).
\]
  Now $\prod_{n<\omega} \Basis^2$ has a natural complete
  metric: the distance between two sequences is $2^{-n}$ when
  $n$ is the index of the first position where the sequences
  differ.  Moreover, $S$ is closed as a subset of
  $\prod_{n<\omega} \Basis^2$, and thus $S$ is a complete
  metric space.
  
  Any sequence $\seq{U_n, V_n}_{n < \omega} \in S$ will have a single
  point of $X$ in $\bigcap_{i} V_i$, because
  $\Strategy$ is a convergent winning strategy for \POINT\ 
  and $X$ is $T_1$. Thus there is a well-defined map
  $f\colon S
  \to X$ such that $f(\seq{U_n,V_n}_{n<\omega})$ is the unique
  element of $\bigcap_{n} V_n$. To check that this
  map is continuous, fix a point $x$ in an open set $U$, and
  a sequence $s \in S$ with $f(s) = x$. Then, because
  $\Strategy$ is convergent, there is some $n$ such that
  $s(n) \subseteq U$. The set of all sequences in $S$ that
  agree with $s$ on the first $n$ coordinates is an open
  neighborhood $\hU$ of $s$ with $f(\hU) \subseteq
  U$.

  To check that $f$ is an open mapping, let $U$ be a basic open set in
  $S$. Without loss of generality, $U$ is determined by a
  finite initial segment $\tau = \seq{U_1, V_1, \ldots, U_n,
  V_n}$ of open
  sets. Now for any $x \in V_n$, there is an extension of
  $\tau$ to an element $s_x$ with $f(s_x) = x$, which is
  obtained by simply playing $x$ and a neighborhood basis of
  $x$ in the Choquet game. Thus $f(U) = V_n$.
  
  Finally, we verify that the weight of the metric space can
  be taken to be the same as the weight of~$X$.  If $X$ is
  finite, then because $X$ is $T_1$ it is discrete, and the
  result is trivial.  Now suppose that $X$ is infinite and
  $\Basis$ is a basis for $X$ of minimal cardinality;
  $\Basis$ will be infinite as well.  Now the space~$S$, if
  constructed from $\Basis$ as above, has a basis whose
  cardinality is no larger than the cardinality of the set
  of finite subsets of $\Basis^2$. This will be exactly the
  cardinality of $\Basis$.
\end{proof}

Our next example shows that the completeness assumption in
Theorem~\ref{T:appl-oimage} cannot be removed altogether. 
We will rely on a characterization of the first-countable
spaces due to Ponomarev.

\begin{theorem}[Ponomarev~\cite{Ponomarev60}]
  \label{T:ponomarev} A $T_0$ space is first-countable if and only if
 it is the open continuous image of a metric space.
\end{theorem}

The space in this example has been discussed by 
Todor{\v{c}}evi{\'c}~\cite{Todor1984}. 

\begin{example}\label{E:noconvergent}
  There is a first-countable Hausdorff space $X$ such that
  there is no convergent strategy for \POINT\ in~$\SCG(X)$.
  Moreover, by Theorem~\ref{T:ponomarev}, this space is the
  open continuous image of a metric space.
\end{example}
\newcommand{\dom}{\operatorname{dom}}

\begin{proof}
  The example relies on several concepts from set theory
  that we define briefly here; these are not used outside of
  the present proof. The set of countable ordinals is
  denoted $\omega_1$. A subset of $\omega_1$ is
  \textit{unbounded} if it has no upper bound less than
  $\omega_1$, and \textit{closed} when it is closed in the
  order topology.  A set is \textit{club} if it is closed and
  unbounded, and \textit{stationary} if it has nonempty
  intersection with every club set.  Every club set is
  stationary, and it is well known that there are stationary
  sets that do not contain any club set.  A function $f$
  from an initial segment of $\omega_1$ to $\omega_1$ is
  \textit{continuous} if it is continuous in the order topology.

  Fix a set $A \subseteq \omega_1$ that is stationary
  and does not contain any club set; thus $\omega_1 \setminus A$
  is unbounded in~$\omega_1$.  We construct our example $X =
  X_A$ as the set of all maximal paths through a certain
  tree~$T$. For each ordinal $\alpha < \omega_1$, let
  $T_\alpha$ consist of all continuous, increasing functions
  from the ordinals less than or equal to $\alpha$ to $\omega_1
  \setminus A$. Then let $T = \bigcup_{\alpha < \omega_1} T_\alpha$.
  We assign $X$ the topology in which each element~$\tau$
  of~$T$ determines a basic open set $N_\tau$, consisting of
  those elements of $X$ that extend~$\tau$. 
  
  Any maximal path $f$ through $T$ can be naturally
  identified with a continuous increasing function from an
  initial segment $\dom(f) \subseteq \omega_1$ to $\omega_1
  \setminus A$.  Moreover, $o(f) = \sup \set{f(\alpha) :
    \alpha \in \dom(f)}$ will be an element of $A$ or will
  be~$\omega_1$. For, if $o(f) < \omega_1$ is not in $A$,
  then we could extend $f$ to a larger continuous increasing
  function, because we have assumed that $\omega_1 \setminus
  A$ is unbounded.
  
  For any $f \in X$, if $o(f) = \omega_1$, then $\dom(f) =
  \omega_1$ and $C = \set{ f(\alpha) : \alpha < \omega_1}$
  will be a club set.  In this case, because $A$ is a
  stationary set, $C \cap A$ is nonempty, contradicting the
  definition of $f$.  Thus, each maximal path $f$ through
  $T$ has a bounded range, and thus a bounded domain, so
  there is some sequence $\seq{\tau(i)}_{i < \omega}$ with
  $f = \bigcap N_{\tau(i)}$. This means that $X$ is first-countable.
  
  Now suppose that $\Strategy$ is a convergent strategy for
  \POINT\ in $\SCG(X)$; we will show that $A$ contains a
  club set, namely the set 
\[
C = \{ o(f) : f \in T \text{ and } o(f) \in A\}.
\]
  As this set is clearly a subset of $A$, we only need
  to prove it is closed and unbounded.
  Note that, because a descending sequence of nonempty open
  sets of $X$ cannot have an empty intersection, $\Strategy$ will
  necessarily be a winning strategy. Moreover, because $X$
  is Hausdorff, any play of $\SCG(X_A)$ that follows
  $\Strategy$ will have a single point in the intersection
  of the open sets played.
  
  To see that $C$ is unbounded, note that for any $\beta <
  \omega_1$, the set of ordinals between $\beta$ and
  $\omega_1$ that are of the form $o(f)$ for some $f \in
  T$ will be a club set, which will have nonempty
  intersection with the stationary set~$A$.
  
  Let $\seq{\alpha_i}_{i < \omega}$ be any increasing
  sequence of elements of $C$. To complete the proof that
  $C$ is a club set, we must show that $\alpha = \sup \alpha_i$ 
  is in~$C$.  We define a play of $\SCG(X)$ that follows
  $\Strategy$. At stage $0$, find some element $f_0 \in X$
  with $o(f_0) = \alpha_0$; this is possible because of the
  definition of $C$. Make \EMPTY\ play $f_0$ and any open
  neighborhood of $f_0$, so that $\Strategy$ returns $V_0$.
  We may assume that $V_0$ is a basic open neighborhood,
  determined by a continuous increasing function $g_0$ from
  an initial segment of $\omega_1$ to $\omega_1 \setminus
  A$. Now, because $f_0 \in V_0$, there is no ordinal in the
  range of $g_0$ that is larger than $\alpha_0$. Thus,
  because $\dom(g_0)$ has a largest element and $\alpha_0$
  is a limit ordinal, we can extend $g_0$ to some $f_1 \in
  V_0$ such that $o(f_1) = \alpha_1$, and then find a basic
  neighborhood $U_1$ of $f_1$ such that every $f \in U_1$
  has $o(f) > \alpha_0$. Let $V_1$ be the response of
  $\Strategy$ when \EMPTY\ now plays $\seq{f_1, U_1}$; then
  $V_1$ is determined by some function $g_1 \in T$.
  Continuing inductively, we generate a play
  $\seq{f_i,U_i}_{i < \omega}$ of $\SCG(X)$ following
  $\Strategy$, and a corresponding sequence $\seq{g_i}_{i <
    \omega}$.
  
  Because $\Strategy$ is a convergent winning strategy,
  there is a unique point $f \in \bigcap_i U_i$. Moreover,
   $f = \bigcup_i g_i$, because otherwise there would be 
  more than one point in $\bigcap_i U_i$.   Now we have \
  $\alpha = \sup \alpha_i = o(f)$, and we proved above
  that $o(f)$ must be in~$A$. By the definition of $C$, this
  implies that $\alpha \in C$, which is what we wanted to prove. Thus,
  if \POINT\ has a convergent strategy in $\SCG(X)$,
   then $A$ contains the club set~$C$. This contradicts
  the assumption that $A$ is stationary but does not contain
  a club set; thus there is no convergent strategy for
  \POINT\ in~$\SCG(X)$.
\end{proof}

Before we prove the remaining theorems from the
introduction, we require two propositions about spaces with
specific kinds of bases.

\begin{proposition}\label{P:bco->convergent}
  If a $T_1$ space $X$ has a basis of countable order, then
  \POINT\ has a stationary convergent strategy in~$\SCG(X)$.
  \textup{(}This strategy may not be a winning
  strategy.\textup{)}
\end{proposition}
\begin{proof}
  Suppose $\Basis$ is a basis of countable order for the
  space~$X$. Define a stationary strategy $\Strategy$ by
  letting $V = \Strategy(x,U)$ be any element of $\Basis$
  such that $x \in V \subsetneq U$, unless that is
  impossible, in which case $U = \Strategy(x,U)$. The second
  case can only occur if $U= \set{x}$.
  
  Suppose that $\seq{x_i, U_i}_{i < \omega}$ is a play of
  $\SCG(X)$ following $\Strategy$. Suppose $x \in \bigcap_i
  U_i$. If the set $\mathcal{U} = \{U_i : i < \omega\}$ is
  infinite, then $\mathcal{U}$ is a neighborhood basis for
  $x$ because $\Basis$ is of countable order. Otherwise,
  $\mathcal{U}$ is finite, in which case there is some $U
  \in \mathcal{U}$ which is a minimal open neighborhood for
  $x$; this also means that $\mathcal{U}$ is a neighborhood
  basis for~$x$.
\end{proof}

The proof of the following proposition is similar to the
proof of Choquet's theorem presented by
Kechris~\cite[sec.~8.E]{Kechris-CDST}

\begin{proposition}\label{P:uniform=metacompact+convergent}
  A $T_1$ space $X$ has a uniform basis if and only if $X$ is
  metacompact and \POINT\ has a stationary convergent
  strategy in~$\SCG(X)$. \textup{(}This strategy may not be
  a winning strategy.\textup{)}
\end{proposition}

\begin{proof}
  For the forward implication, suppose $\Basis$ is a uniform
  basis for~$X$. We know from
  Theorem~\ref{P:uniform->metacompact} that $X$ is
  metacompact. Also, $\Basis$ is of countable order, which
  implies that \POINT\ has a stationary convergent strategy
  in $\SCG(X)$ by Proposition~\ref{P:bco->convergent}.
  
  For the reverse implication, suppose that $X$ is
  metacompact and that $\Strategy$ is a stationary
  convergent strategy for \POINT\ in~$\SCG(X)$.
  
  We use $\Strategy$ to inductively construct a uniform
  basis $\Basis$ for $X$ in $\omega$ stages. At stage $0$, set
  $\Basis_0 = \set{X}$, which is a point-finite open cover
  of~$X$. At stage $i+1$, having defined the point-finite
  open cover $\Basis_i$, pick $\Basis_{i+1}$ to be a
  point-finite open refinement of the family
  $\set{\Strategy(x,U) : x \in U \in \Basis_i}$.  We claim that
  $\Basis = \bigcup_{i<\omega} \Basis_i$ is a uniform basis for~$X$.

  To see this, suppose that $\mathcal{A} \subseteq \Basis$ is
  infinite and that $x \in \bigcap \mathcal{A}$. Let
  $\widehat{\mathcal{A}}$ be the upward closure of
  $\mathcal{A}$ in~$\Basis$. Now $\widehat{\mathcal{A}}$ is
  a basis at $x$ only if $\mathcal{A}$ is a basis at~$x$.
  Since $\Basis_i$ is point-finite and $x$ belongs to every
  element of $\widehat{\mathcal{A}}$, the set
  ${\widehat{\mathcal{A}}} \cap \Basis_i$ is finite for 
  each~$i < \omega$.

  Consider the finitely branching tree $\mathcal{T}$ of finite
  sequences $\seq{U_i}_{i \leq n}$ such that $U_i \in
  \widehat{\mathcal{A}} \cap \Basis_i$ and, if $i \geq 1$, then $U_i
  \subseteq \Strategy(w,U_{i-1})$ for some $w \in U_i$.  The
  definition of $\seq{\Basis_i}_{i<\omega}$ guarantees that every
  element of $\widehat{\mathcal{A}}$ belongs to a sequence in
  $\mathcal{T}$. Since $\widehat{\mathcal{A}}$ is infinite it follows
  that $\mathcal{T}$ is also infinite.

  By K{\"o}nig's Lemma, $\mathcal{T}$ has an infinite branch
  $\seq{U_i}_{i<\omega}$. By definition of $\mathcal{T}$, we can pick
  a sequence of points $\seq{w_i}_{i<\omega}$ such that
  $\seq{w_i,U_i}_{i<\omega}$ is a play of the game $\SCG(X)$
  following~$\Strategy$. Since $\Strategy$ is convergent,
  it follows that $\set{U_i}_{i<\omega} \subseteq
  \widehat{\mathcal{A}}$ is a neighborhood basis at $x$, hence
  $\mathcal{A}$ is also a neighborhood basis at~$x$.
\end{proof}

The next theorem parallels Theorem~\ref{T:appl-oimage}; it
characterizes the $T_1$ metacompact spaces for which \POINT\ 
has a convergent strategy in the Choquet game.

\begin{theorem}
  \label{T:appl-okimage}
  Let $X$ be a $T_1$ space. Then $X$ is the open continuous compact
  image of a metric space if and only if $X$ is metacompact
  and \POINT\ has a convergent strategy
  in~$\SCG(X)$. Moreover, if these conditions hold, then
  \POINT\ has a stationary convergent strategy
  in~$\SCG(X)$. \textup{(}This strategy may not be a winning
  strategy.\textup{)}
\end{theorem}

\begin{proof}
  Let $X$ be a $T_1$ space. By
  Theorem~\ref{T:arhangelskii}, $X$ is the open continuous compact image
  of a metrizable space if and only if $X$ has a uniform
  basis.   But, by
  Proposition~\ref{P:uniform=metacompact+convergent}, $X$
  has a uniform basis if and only if $X$ is metacompact and
  \POINT\ has a stationary convergent strategy in~$\SCG(X)$.
\end{proof}

Our next theorem will require an additional result on metacompactness.

\begin{proposition}\label{P:paracompact->metacompact}
  Let $X$ and $Y$ be $T_1$ spaces and let $f$ be an open continuous
  compact mapping of $X$ onto~$Y$. If $X$ is paracompact then $Y$ is
  metacompact.
\end{proposition}

\begin{proof}
  Let $\mathcal{V}_0$ be an open cover of $Y$ and let $\mathcal{U}_0 =
  \set{f^{-1}(V): V \in \mathcal{V}_0}$, which is an open cover of
  $X$. Since $X$ is paracompact, we can find a locally-finite open
  refinement $\mathcal{U}$ of $\mathcal{U}_0$. We claim that
  $\mathcal{V} = \set{f(U) : U \in \mathcal{U}}$ is a point-finite
  open refinement of~$\mathcal{V}_0$.

  The fact that $\mathcal{V}$ is an open refinement of $\mathcal{V}_0$
  is clear. Since $f$ is open and onto, $\mathcal{V}$ is certainly an
  open cover of $Y$. For each $U \in \mathcal{U}$ there is a $V \in
  \mathcal{V}_0$ such that $U \subseteq f^{-1}(V)$. Therefore, $f(U)
  \subseteq V$, which shows that $\mathcal{V}$ is a
  refinement of~$\mathcal{V}_0$.

  It remains to show that $\mathcal{V}$ is point-finite. Fix $y \in
  Y$.  For each $x \in X$, we can find an open neighborhood $W_x$ of
  $x$ that meets only finitely many elements of $\mathcal{U}$.  Because
  $f$ is a compact mapping, $f^{-1}(y)$ is a compact subset of~$X$. So
  we can find $x_1,\dots,x_k \in f^{-1}(y)$ such that $f^{-1}(y)
  \subseteq W_{x_1} \cup \cdots \cup W_{x_k}$. If $U \in \mathcal{U}$
  and $y \in f(U)$ then $U \cap W_{x_i} \neq \varnothing$ for some $i
  \in \set{1,\dots,k}$. By our choice of $W_{x_i}$, there are only
  finitely many such $U \in \mathcal{U}$ for each $i \in
  \set{1,\dots,k}$. Therefore, the set $\set{V \in \mathcal{V}: y \in
    V}$ is finite.
\end{proof}

Our final theorem gives a characterization of the $T_1$
metacompact spaces for which  \POINT\ has a convergent winning
strategy for the Choquet game.

\begin{theorem}\label{T:appl-okcimage}
  Let $X$ be a $T_1$ space. Then $X$ is the open continuous compact
  image of a complete metric space if and only if $X$ is
  metacompact and \POINT\ has a convergent winning strategy
  in~$\SCG(X)$.  
Moreover, the strategy can be taken to be stationary, and the
metric space can be taken to have the same weight as~$X$.
\end{theorem}

\begin{proof}
%
First, assume that $f \colon Z \to X$ is an open continuous 
compact surjection from a complete metric space $Z$ to
$X$. It follows from 
Theorem~\ref{T:appl-oimage} that \POINT\ has a convergent
winning strategy in $\SCG(X)$, and it follows from
Proposition~\ref{P:paracompact->metacompact} that $X$ is
metacompact. The proof of the converse implication will show
that we can take the strategy to be stationary.

For the converse, assume $X$ is a metacompact $T_1$ space
and \POINT\ has a convergent winning strategy for~$\SCG(X)$.
By Propositions~\ref{P:convergent-stationary}
and~\ref{P:uniform=metacompact+convergent}, there
is a uniform basis $\Basis$ for $X$. Because $\Basis$ is
open-finite, we know that \POINT\ has a stationary winning
strategy $\Strategy$ in $\SCG(X,\Basis)$, by
Theorem~\ref{T:open-finite-stationary}
and~Proposition~\ref{P:changeofbasis}.  Because $\Basis$ is
of countable order, this strategy will also be convergent.
We may assume, as in the proof of
Proposition~\ref{P:uniform=metacompact+convergent}, that
$\Basis = \bigcup_{n < \omega} \Basis_n$, where $\Basis_0 = \{X\}$ and
each $\Basis_{n+1}$ is a point-finite refinement of
$\{\Strategy(x, U) : x \in U \in \Basis_n\}$.

The proof now resembles the second part of the proof of
Theorem~\ref{T:appl-oimage}. The space $\prod_{n<\omega}
\Basis_n$ has a complete metric in which the distance
between two distinct sequences is $2^{-n}$ when $n$ is the index of
the first position where the sequences differ.
Let $S$ be the set of all descending sequences
$\seq{U_n}_{n<\omega} \in \prod_{n<\omega} \Basis_n$ of open
subsets of $X$ such that for every $n < \omega$, there is some 
$x \in U_{n+1}$ such that $U_{n+1} \subseteq \Strategy(x,
U_n)$. Then $S$ is a closed subset of $\prod_{n<\omega}
\Basis_n$, and thus $S$ is a complete metric space.
  
Now let $f:S \to X$ be the unique map such that
$\set{f(\seq{U_n}_{n<\omega})} = \bigcap_{n<\omega} U_n$. It
is easy to check that this is an open continuous mapping from
$S$ onto $X$, using the same technique as the proof of
Theorem~\ref{T:appl-oimage}. To see that $f$ is a compact
mapping, note that $f^{-1}(x)$ is a closed subset of $S_x$,
where
\[
S_x  =
\prod_{n<\omega} \set{U \in \Basis_n: x \in U}.
\]
 Because
each $\Basis_n$ is point-finite, $S_x$ is homeomorphic to a
product of finite discrete spaces and is thus compact.
\end{proof}

\bibliographystyle{amsalpha} 
\bibliography{choquet3}

\noindent \textbf{Fran{\c c}ois G. Dorais}\\
Department of Mathematics\\
University of Michigan\\
530 Church Street\\
Ann Arbor MI 48109\quad USA
\vspace{\baselineskip}

\noindent \textbf{Carl Mummert}\\
Department of Mathematics\\
Marshall University\\
1 John Marshall Drive\\
Huntington, WV 25755\quad USA

\end{document}